\documentclass[11pt]{article}
\usepackage{amsmath}
\usepackage{amssymb}
\usepackage{amscd}
\usepackage{amsmath}
\usepackage{setspace}
\usepackage{graphics,epsfig}
\usepackage{graphicx}
\usepackage{float}

\DeclareMathAlphabet{\eusm}{OT1}{eusm}{m}{n}
\newtheorem{thm}{Theorem}[section]

\newtheorem{exam}[thm]{Example}
\newtheorem{rem}[thm]{Remark}
\newtheorem{lem}[thm]{Lemma}

\newcommand{\QQ}{\mathbb{Q}}
\newcommand{\ZZ}{\mathbb{Z}}
\newcommand{\NN}{\mathbb{N}}

\DeclareMathOperator{\Ker}{Ker}
\DeclareMathOperator{\image}{Im}

\newcommand{\ess}{\subseteq_e}
\newcommand{\coess}{\subseteq_c}

\newcommand{\graph}[1]{\langle #1 \rangle}

\def\vsp{\vspace{1ex}}

\newenvironment{proof}{\par\noindent{\bf Proof \,}}{$\hfill
\Box$\par\bigskip} \textheight 23cm
\begin{document}
\begin{center}
{\rm {\LARGE On lifting modules \\
which do not satisfy the finite internal exchange property}}
\end{center}
\vsp

\begin{center}
{\rm {\Large  Yoshiharu Shibata}}\\
Graduate School of Sciences and Technology for Innovation,
Yamaguchi University,\\
1677-1 Yoshida, Yamaguchi, 753-8512, Japan\\
e-mail: b003wb@yamaguchi-u.ac.jp
\end{center}

\vsp

\begin{center}
{\large {\bf Abstract}}
\end{center}

A module $M$ is said to be {\it lifting} if, 
for any submodule $N$ of $M$, 
there exists a direct summand $X$ of $M$ contained in $N$ 
such that $N/X$ is small in $M/X$. 
A module $M$ is said to satisfy 
the {\it finite internal exchange property} if, 
for any direct summand $X$ of $M$ and
any finite direct sum decomposition $M = \bigoplus_{i = 1}^n M_i$, 
there exists $M_i' \subseteq M_i$ $(i = 1, 2, \ldots, n)$ such that 
$M = X \oplus (\bigoplus_{i = 1}^n M_i')$. 
In this paper, we consider the open problem: 
does any lifting module satisfy 
the finite internal exchange property?
We give characterizations for
the square of a hollow and uniform module 
to be lifting, 
and solve the above problem negatively 
as an application of this result.

%It is characterized by ``the generalized projectivity'' 
%that the direct sum of lifting modules is lifting 
%under the condition of the finite internal exchange property. 
%However, no characterizations by some kind of projectivity have given
%if the finite internal exchange property is not assumed. 
%On the contrary, an example of a lifting module 
%without the finite internal exchange property is not found. 
%In this paper, we give a projectivity 
%for the square of a hollow and uniform module to be lifting, 
%and make an example of a lifting module without 
%the finite internal exchange property. 

\footnote[0]{$2010$ Mathematics Subject Classification{\rm :} 
Primary  16D40, ; Secondary 16D70. \\
\hspace{5mm}
Key Words{\rm:} lifting modules, extending modules, finite internal exchange property.}

%%%%%%%%%%%%%%%%%%% 1 %%%%%%%%%%%%%%%%%%%%%
\section{Preliminaries}

Throughout this paper, $R$ is a ring with identity and 
modules are unitary right $R$-modules. 
Let $M$ be a module and $N, K$ submodules of $M$ with $K \subseteq N$. 
$N$ is said to be {\it small} in $M$ (or a {\it small submodule} of $M$) 
if $N + X \neq M$ for any proper submodule $X$ of $M$
and we denote by $N \ll M$ in this case. 
A pair $(Q, f)$ of a module $Q$ and an epimorphism $f : Q \to M$ 
be a {\it small cover} of $M$ if $\Ker f \ll Q$. 
$K$ is said to be a {\it coessential submodule} of $N$ in $M$ if 
$N/K \ll M/K$ 
and we write $K \coess N$ in $M$ in this case. 

%$N$ is said to be a {\it coclosed submodule} of $M$ or {\it coclosed} in $M$, 
%if $N$ has no proper coessential submodule in $M$. 
%Let $\{ M_i \ | \ i \in I \}$ be a family of modules. 
%The direct sum decomposition $M = \bigoplus_{i \in I} M_i$ is said to be 
%{\it exchangeable} if, for any direct summand $X$ of $M$, 
%there exists $M_i' \subseteq M_i$ $(i \in I)$ such that 
%$M = X \oplus (\bigoplus_{i \in I} M_i')$. 
%A module $M$ is said to have the {\it finite internal exchange property} 
%(or briefly, {\it FIEP}) if 
%any finite direct sum decomposition $M = \bigoplus_{i = 1}^n M_i$ 
%is exchangeable. 
A module $M$ is said to satisfy the {\it finite internal exchange property} 
(or briefly, {\it FIEP}) if, for any direct summand $X$ of $M$ and
any finite direct sum decomposition $M = \bigoplus_{i = 1}^n M_i$, 
there exists $M_i' \subseteq M_i$ $(i = 1, 2, \ldots, n)$ such that 
$M = X \oplus (\bigoplus_{i = 1}^n M_i')$. 
Let $M = A \oplus B$ be a module 
and $h : A \to B$ a homomorphism. 
Then $\{a + h(a) \ | \ a \in A\}$ 
is called a {\it graph} of $h$ and denoted by $\graph{h}$. 
It is clear that $M = \graph{h} \oplus B$, 
$M = A + \graph{h}$ if $h$ is an epimorphism, and $A \cap \graph{h} = \Ker h$. 

A module $M$ is said to be {\it extending} (or {\it CS}) if, 
for any submodule $N$ of $M$, 
there exists a direct summand $X$ of $M$ 
such that $N$ is an essential submodule of $X$. 
An indecomposable extending module is called {\it uniform}. 
A lifting module is defined as a dual concept of an extending module, 
that is, a module $M$ is said to be {\it lifting} if, 
for any submodule $N$ of $M$, 
there exists a direct summand $X$ of $M$ 
such that $X$ is a coessential submodule of $N$ in $M$. 
An indecomposable lifting module is called {\it hollow}. 
It is well-known that uniform modules (hollow modules, resp.) 
are closed under nonzero submodules and essential extensions 
(nonzero factor modules and small covers, resp.). 
A module $M$ is said to be {\it uniserial} if 
its submodules are linearly ordered by inclusion. 
Clearly, any uniserial module is hollow and uniform. 
However the converse is not true. 
We consider 
\[
	R = \left(\begin{array}{@{\,}cccc@{\,}}
			K & K & K & K \\
			0 & K & 0 & K \\
			0 & 0 & K & K \\
			0 & 0 & 0 & K
	\end{array}\right), \quad
	M_R = (K, K, K, K)
\]
where $K$ is a field. 
Then $M$ has only $6$ submodules 
\[
	M, \ (0, K, K, K), \ (0, K, 0, K), \ (0, 0, K, K), \ (0, 0, 0, K), \ 0.
\]
Hence $M$ is hollow and uniform but not uniserial. 

Extending modules and lifting modules are important 
because they characterize right noetherian rings, 
right perfect rings, 
semiperfect rings, right (co-)H-rings and 
Nakayama rings (cf. \cite{BO}).

Let $A$ and $B$ be modules. 
$A$ is called {\it generalized $B$-injective} if, 
ffor any module $X$, 
any homomorphism $f : X \to A$ 
and any monomorphism $g : X \to B$, 
there exist direct sum decompositions 
$A = A_1 \oplus A_2$ and $B = B_1 \oplus B_2$, 
a homomorphism $h_1 : B_1 \to A_1$ and a monomorphism $h_2 : A_2 \to B_2$ 
such that $p_1f = h_1q_1g$ and $q_2g = h_2p_2f$, 
where $p_i : A = A_1 \oplus A_2 \to A_i$, 
$q_i : B = B_1 \oplus B_2 \to B_i$ $(i = 1, 2)$ are canonical projections (\cite{HKO}). 
$A$ is called {\it generalized $B$-projective} if, 
for any module $X$, 
any homomorphism $f : A \to X$ 
and any epimorphism $g : B \to X$, 
there exist direct sum decompositions 
$A = A_1 \oplus A_2$ and $B = B_1 \oplus B_2$, 
a homomorphism $h_1 : A_1 \to B_1$ and an epimorphism $h_2 : B_2 \to A_2$ 
such that $f|_{A_1} = gh_1$ and $g|_{B_2} = fh_2$. 
It is already known that, 
a finite direct sum of lifting modules (extending modules, resp.) 
 with the FIEP $M = \bigoplus_{i = 1}^n M_i$ is 
 lifting (extending, resp.) with the FIEP if and only if 
$K$ and $L$ are 
relative generalized projective
(relative generalized injective, resp.) 
for every $k = 1, 2, \ldots, n$, 
any direct summand $K$ of $M_k$ and 
any direct summand $L$ of $\bigoplus_{i \neq k} M_i$ 
by \cite[Theorem 3.7]{K1} (\cite[Theorem 2.15]{HKO}, resp.). 
%However there are no results of removing the condition ``FIEP''. 
%Also, an example of a lifting module 
%without the FIEP is not known yet. 

%In this paper, 
%we solve the problem: 
%does any lifting module satisfy the FIEP?, 
%and give characterizations  
%for the square of a hollow and uniform module 
%to be lifting (extending) 
%by the certain projectivity (an injectivity). 

In this paper, we consider the open problem: 
does any lifting module satisfy 
the finite internal exchange property?
By certain projectivities (injectivities), 
we give characterizations for
the square of a hollow and uniform module 
to be a lifting module (an extending module) 
which does not necessarily satisfy the FIEP.
Using this result, we give an example of a lifting module 
not satisfying the FIEP 
in order to solve the above problem negatively.

For undefined terminologies, the reader is referred to 
\cite{AF}, \cite{BO}, \cite{CLVW}, \cite{MM} and \cite{Wi}.

%%%%%%%%%%%%%%%%%%%%%%%%%%%%%%%%%%%%%%%%%%
\section{Main results}

\begin{lem}\label{hollowlemma}
	Let $A$ and $B$ be modules and put $M = A \oplus B$. 
    For any nonzero proper direct summand $X$ of $M$, the following holds:
	\begin{itemize}
    	\item[(1)]	If $A$ and $B$ are hollow, then so is $X$. 
        \item[(2)]	If $A$ and $B$ are uniform, then so is $X$. 
    \end{itemize}
\end{lem}

\begin{proof}
	Let $p : M = A \oplus B \to A$ and $q : M = A \oplus B \to B$ 
    be canonical projections. 
    
	(1) Since $A$ and $B$ are hollow and $X$ is non-small, 
    	$X$ satisfies either $p(X) = A$ or $q(X) = B$. 
        Without loss of generality, we can take $X$ with $p(X) = A$. 
        By $X \neq M$, 
        we see $X \cap B \ll B$ because $B$ is hollow. 
        Since $X$ is a proper direct summand of $M$, 
        we obtain $\Ker p|_X = X \cap B \ll X$. 
        Hence $(X, p|_X)$ is a small cover of $A$. 
        Therefore $X$ is hollow. 
    
    (2) Since $A$ and $B$ are uniform and $X$ is non-essential, 
    	$X$ satisfies either $X \cap A = 0$ or $X \cap B = 0$. 
        Without loss of generality, we can take $X$ with $X \cap A = 0$. 
        Then $q|_X : X \to B$ is a nonzero monomorphism. 
        Therefore $X$ is uniform because 
        it is isomorphic to a submodule of a uniform module $B$. 
\end{proof}

Here we give a key lemma in this paper. 

\begin{lem}\label{graphlemma}
	Let $U$ be a hollow and uniform module
    and put $M = U^2$, $U_1 = U \times 0$ and $U_2 = 0 \times U$. 
    Then for any submodule $N_1$ of $U_1$ and 
    any epimorphism $h_1$ from $N_1$ to $U_2$, 
    $\langle h_1 \rangle$ is a direct summand of $M$. 
\end{lem}

\begin{proof}
    If $N_1 = U_1$ or $\Ker h_1 = 0$, 
    it is clear $M = \langle h_1 \rangle \oplus U_2$ 
    or $M = \langle h_1 \rangle \oplus U_1$. 
    We assume $N_1 \neq U_1$ and $\Ker h_1 \neq 0$, 
    and take a submodule $N_2$ of $U_2$ 
    which is a natural isomorphic image of $N_1$ and 
    an epimorphism $h_2$ from $N_2$ to $U_1$. 
    Now we prove $M = \langle h_1 \rangle \oplus \langle h_2 \rangle$. 
    
    First we show $M = \langle h_1 \rangle + \langle h_2 \rangle$. 
    Let $\iota_i : h_i^{-1}(N_j) \to U_i$ $(i \neq j)$ 
    be the inclusion mapping. 
    Then $\image \iota_i = h_i^{-1}(N_j) \subseteq h_i^{-1}(U_j) 
    = N_i \subsetneq U_i$
    $(i \neq j)$. 
    We define a homomorphism $h_i'$ from $h_i^{-1}(N_j)$ to $U_i$ 
    by $h_i'(x) = h_jh_i(x)$ for $x \in h_i^{-1}(N_j)$ $(i \neq j)$. 
    Then $h_i'$ is onto $(i = 1, 2)$. 
    Since $U_i$ is hollow, 
    we obtain that $\iota_i - h_i' : h_i^{-1}(N_j) \to U_i$ is onto
    $(i \neq j)$. 
    For any element $u_1 + u_2$ of $M$ $(u_i \in U_i)$, 
    there exists an element $x_i$ of $h_i^{-1}(N_j)$ such that 
    $(\iota_i - h_i')(x_i) = u_i$ $(i \neq j)$. 
    Hence 
    \begin{align*}
    	u_1 + u_2 	&= ((x_1 - h_2(x_2)) + h_1(x_1 - h_2(x_2))) 
        			+ ((x_2 - h_1(x_1)) + h_2(x_2 - h_1(x_1))) \\
	    &\in \langle h_1 \rangle + \langle h_2 \rangle. 
    \end{align*}
    Therefore  $M = \langle h_1 \rangle + \langle h_2 \rangle$. 
    
    Next we show $\langle h_1 \rangle \cap \langle h_2 \rangle = 0$. 
    We see 
    \[
    	(\langle h_1 \rangle \cap \langle h_2 \rangle) \cap \Ker h_1 
    	= (\langle h_1 \rangle \cap \langle h_2 \rangle) 
        		\cap (\langle h_1 \rangle \cap N_1) 
        \subseteq \langle h_2 \rangle \cap N_1 = 0. 
    \]
	Since $\langle h_1 \rangle \cong N_1$ is uniform and $\Ker h_1 \neq 0$, 
    we obtain $\langle h_1 \rangle \cap \langle h_2 \rangle = 0$. 
\end{proof}

The following is one of our main results. 
%Such projectivity is a generalization 
%of the generalized self-projectivity 
%limited to a hollow and uniform module. 

\begin{thm}\label{lifting}
	Let $U$ be a hollow and uniform module 
    and  put $M = U^2$, $U_1 = U \times 0$ and $U_2 = 0 \times U$. 
    Then the following conditions are equivalent: 
    \begin{itemize}
    	\item[(a)]	$M$ is lifting, 
    	\item[(b)]	for any module $X$, 
        			any homomorphism $f : U_1 \to X$ 
        			and any epimorphism $g : U_2 \to X$,                     one of the following holds: 
                    \begin{itemize}
				    	\item[(i)]	there exists a homomorphism 
                        			$h : U_1 \to U_2$ 
				        			such that $f = gh$, 
				        \item[(ii)]	there exist a submodule $N$ of $U_2$ 
                        			and an epimorphism $h : N \to U_1$ 
                                    such that $g|_N = fh$, 
   					\end{itemize}
        \item[(c)]	for any module $X$, 
        			any homomorphism $f : U_1 \to X$ 
        			and any epimorphism $g : U_2 \to X$, 
                    one of the following holds: 
                    \begin{itemize}
				    	\item[(i)]	there exists a homomorphism 
                        			$h : U_1 \to U_2$ 
				        			such that $f = gh$, 
				        \item[(ii)] there exist a submodule $K$ of $\Ker g$									 and a monomorphism $h : U_1 \to U_2/K$ 
                                     such that $g'h = f$, where 
                                     $g' : U_2/K \to X$ is defined by 
                                     $g'(\overline{u}) = g(u)$ 
                                     for $\overline{u} \in U_2/K$.
   					\end{itemize}
    \end{itemize}
\end{thm}

\begin{proof}
    Let $p_i : M = U_1 \oplus U_2 \to U_i$ 
    be the canonical projection ($i = 1, 2$). 
	
    (a) $\Rightarrow$ (b): 
    Let $f : U_1 \to X$ be a nonzero homomorphism and 
    $g : U_2 \to X$ an epimorphism. 
    We define a homomorphism $\varphi : M \to X$ 
    by $\varphi(u_1 + u_2) = f(u_1) - g(u_2)$ 
    for $u_i \in U_i$ $(i = 1, 2)$. 
    Since $M$ is lifting, there exists a direct summand $A$ of $M$ 
    such that $A \coess \Ker \varphi$ in $M$. 
    Then $M = \Ker \varphi + U_2 = A + U_2$ because $g$ is onto. 
    So $p_1(A) = U_1$. 
    
    If $A \cap U_2 = 0$, 
    we can define a homomorphism $h : U_1 = p_1(A) \to U_2$ 
    by $h(p_1(a)) = p_2(a)$ for $a \in A$, and 
    $h$ satisfies $f = gh$. 
    Therefore (i) holds. 
    
    Otherwise we see $A \cap U_1 = 0$ since $U$ is uniform. 
    Hence we can define an epimorphism 
    $h : p_2(A) \to p_1(A) = U_1$ by $h(p_2(a)) = p_1(a)$ for $a \in A$, 
    and $h$ satisfies $g|_{p_2(A)} = fh$. 
    Therefore (ii) holds.

    (b) $\Rightarrow$ (a): 
	Let $X$ be a submodule of $M$. 
    We may assume that $X$ is a proper non-small submodule of $M$. 
    Since $U_1$ and $U_2$ are hollow with $U_1 \cong U_2$, 
    we only consider the case $p_1(X) = U_1$. 
    Then $M = X + U_2$. 
    Let $\pi : M \to M/X$ be the natural epimorphism. 
    Since $\pi|_{U_2}$ is onto, 
	one of the following (i) or (ii) holds: 
    \begin{itemize}
    	\item[(i)]	there exists a homomorphism $h : U_1 \to U_2$ 
        				such that $\pi|_{U_1} = \pi|_{U_2}h$, 
        \item[(ii)]	there exist a submodule $N$ of $U_2$ and 
        				an epimorphism $h : N \to U_1$ such that 
                        $\pi|_N = \pi|_{U_1}h$. 
    \end{itemize}
    In either case, we see
    $\graph{-h}$ is a direct summand of $M$ by Lemma \ref{graphlemma}, 
    and $\graph{-h} \subseteq X$ by the commutativity of the diagram. 
    Put $M = \graph{-h} \oplus T$ using a direct summand $T$ of $M$. 
    Since $T$ is hollow by Lemma \ref{hollowlemma}, 
    we obtain $T \cap X \ll T$. 
    Hence $\graph{-h} \coess X$ in $M$. 
    Therefore $M$ is lifting.

    (b) $\Rightarrow$ (c): 
    It is enough to show (b)(ii) $\Rightarrow$ (c)(ii). 
    For any homomorphism $f : U_1 \to X$ and 
    any epimorphism $g : U_2 \to X$, 
    we assume that 
    there exist a submodule $N$ of $U_2$ and 
    an epimorphism $h : N \to U_1$ 
    such that $g|_N = fh$. 
    Then $\Ker h \subseteq \Ker g$, 
    hence we can define an epimorphism $g' : U_2/\Ker h \to X$ 
    by $g'(\overline{u}) = g(u)$ for $\overline{u} \in U_2/\Ker h$. 
    Let $\overline{h} : N/\Ker h \to U_1$ be the natural isomorphism 
    and $\iota : N/\Ker h \to U_2/\Ker h$ the inclusion mapping, 
    and put $h' = \iota \overline{h}{}^{-1}$. 
    Clearly, $h'$ is a monomorphism and $g'h' = f$.

    (c) $\Rightarrow$ (b): 
    We show (c)(ii) $\Rightarrow$ (b)(ii). 
    For any homomorphism $f : U_1 \to X$ and 
    any epimorphism $g : U_2 \to X$, 
    we assume that 
    there exist a submodule $K$ of $\Ker g$ 
    and a monomorphism $h : U_1 \to U_2/K$ 
    such that $f = g'h$, 
    where $g' : U_2/K \to X$ is defined 
    by $g'(\overline{u}) = g(u)$ for $\overline{u} \in U_2/\Ker h$. 
    We express $\image h = N/K$. 
    Let $\varphi : N/K \to U_1$ be the inverse map of $h$ and
    $\pi : N \to N/K$ the natural epimorphism, 
    and put $h' = \varphi \pi$. 
    Then $h'$ is onto and $g|_N = fh'$. 
\end{proof}

\begin{thm}\label{extending}
	Let $U$ be a uniform and hollow module 
    and put $M = U^2$, $U_1 = U \times 0$ and $U_2 = 0 \times U$. 
    Then the following conditions are equivalent: 
    \begin{itemize}
    	\item[(a)]	$M$ is extending, 
        \item[(b)] 	for any module $X$, 
        			any homomorphism $f : X \to U_2$ and 
        			any monomorphism $g : X \to U_1$, 
				    one of the following holds: 
				    \begin{itemize}
				    	\item[(i)]	there exists a homomorphism 
                        			$h : U_1 \to U_2$ such that $f = hg$, 
				        \item[(ii)]	there exist a submodule $K$ of $U_1$ 
                        			and a monomorphism $h : U_2 \to U_1/K$ 
                                    such that $hf = \pi g$, 
                                    where $\pi$ is the natural epimorphism 
                                    from $U_1$ to $U_1/K$, 
    				\end{itemize}
        \item[(c)]	for any module $X$, 
        			any homomorphism $f : X \to U_2$ and 
        			any monomorphism $g : X \to U_1$, 
				    one of the following holds: 
				    \begin{itemize}
				    	\item[(i)]	there exists a homomorphism 
                        			$h : U_1 \to U_2$ such that $f = hg$, 
				        \item[(ii)]	there exist a submodule $N$ of $U_1$ 
                        			containing $\image g$ and 
				        			an epimorphism $h : N \to U_2$ 
                                    such that $f = hg$. 
    				\end{itemize}
    \end{itemize}
\end{thm}

\begin{proof}
    Let $p_i : M = U_1 \oplus U_2 \to U_i$ 
    be the canonical projection ($i = 1, 2$). 
    
    (a) $\Rightarrow$ (c): 
    Let $f : X \to U_2$ be a nonzero homomorphism and 
    $g : X \to U_1$ a monomorphism. 
    We define a homomorphism $\varphi : X \to M$ 
    by $\varphi(x) = g(x) + f(x)$ for $x \in X$. 
    Since $M$ is extending, 
    there exists a direct summand $A$ of $M$ 
    such that $\image \varphi \ess A$. 
    By $\image \varphi \cap U_2 = 0$, $A \cap U_2 = 0$. 
    
    If $p_1(A) = U_1$, 
    we can define a homomorphism $h : U_1 = p_1(A) \to U_2$ 
    by $h(p_1(a)) = p_2(a)$ for $a \in A$, 
    and $h$ satisfies $f = hg$. 
    Therefore (i) holds. 
    
    Otherwise, we see $p_2(A) = U_2$ since $U$ is hollow. 
    We see $\image g \subseteq p_1(A)$, and
    we can define an epimorphism 
    $h : p_1(A) \to p_2(A) = U_2$ by $h(p_1(a)) = p_2(a)$ for $a \in A$. 
    Then $h$ satisfies $f = hg$. 
    Therefore (ii) holds.

	(c) $\Rightarrow$ (a): 
    Let $X$ be a submodule of $M$. 
    We may assume that 
    $X$ is a nonzero non-essential submodule of $M$. 
    Since $U_1$ and $U_2$ are uniform with $U_1 \cong U_2$, , 
	we only consider the case $X \cap U_2 = 0$ because $U$ is uniform. 
    Since $p_1|_X$ is a monomorphism, 
	one of the following (i) or (ii) holds: 
    \begin{itemize}
    	\item[(i)]	there exists a homomorphism $h : U_1 \to U_2$ 
        				such that $p_2|_X = hp_1|_X$. 
        \item[(ii)]	there exist a submodule $N$ of $U_1$ 
				        containing $p_1(X)$ and 
        				an epimorphism $h : N \to U_2$ such that 
                        $p_2|_X = hp_1|_X$. 
    \end{itemize}
    In either case, 
    $\graph{h}$ is a direct summand of $M$ by Lemma \ref{graphlemma}, 
    and $X \subseteq \graph{h}$ by commutativity of the diagram. 
    Since $\graph{h}$ is uniform by Lemma \ref{hollowlemma}, 
    we obtain $X \ess \graph{h}$. 
    Therefore $M$ is extending.

	(c) $\Rightarrow$ (b): 
    It is enough to show (c)(ii) $\Rightarrow$ (b)(ii). 
    For any homomorphism $f : X \to U_2$ and 
    any monomorphism $g : X \to U_1$, 
    we assume that 
    there exist a submodule $N$ of $U_1$ containing $\image g$ and 
    an epimorphism $h : N \to U_2$ such that $f = hg$. 
    Let $\overline{h} : N/\Ker h \to U_2$ be the natural isomorphism 
    and $\iota : N/\Ker h \to U_1/\Ker h$ the inclusion mapping, 
    and put $h' = \iota \overline{h}{}^{-1}$. 
    Then $h'$ is a monomorphism and $h'f = \pi g$, where 
    $\pi : U_1 \to U_1/\Ker h$ is the natural epimorphism.

    (b) $\Rightarrow$ (c): 
    We show (b)(ii) $\Rightarrow$ (c)(ii). 
    For any homomorphism $f : X \to U_2$ and 
    any monomorphism $g : X \to U_1$, 
    we assume that there exist a submodule $K$ of $U_1$ 
    and a monomorphism $h : U_2 \to U_1/K$ 
    such that $hf = \pi g$, 
    where $\pi : U_1 \to U_1/K$ is the natural epimorphism. 
    We express $\image h = N/K$. 
    Let $\varphi : N/K \to U_2$ be the inverse map of $h$ 
    and $\eta : N \to N/K$ the natural epimorphism, 
    and put $h' = \varphi \eta$. 
    Then we see $\image g \subseteq N$, 
    $h'$ is an epimorphism and $f = h'g$. 
\end{proof}

\begin{rem}
	In Theorem \ref{extending}, 
    the assumption ``hollow'' cannot be removed. 
    In fact, $U_\ZZ = \ZZ$ is uniform and $U^2$ is extending. 
    However it does not hold neither (i) nor (ii) 
    in Theorem \ref{extending} (b) 
    for a homomorphism $f : 2\ZZ \to U$ defined by $f(2n) = 3n$ 
    and the inclusion mapping $g : 2\ZZ \to U$. 
\end{rem}

Lifting modules do not necessarily satisfy the FIEP. 
We can make an example of a lifting module without the FIEP, 
using Theorem \ref{lifting}. 

\begin{exam}
    Let $\ZZ_{(p)}$ and $\ZZ_{(q)}$ be the localizations of $\ZZ$ 
    at two distinct prime numbers $p$ and $q$ respectively. 
    We consider a semiperfect ring 
    $R = \left(\begin{array}{@{\,}cc@{\,}}
				\ZZ_{(p)} & \QQ \\
				0 & \ZZ_{(q)}
		\end{array}\right)$
    and its right ideal 
    $L = \left(\begin{array}{@{\,}cc@{\,}}
				0 & \ZZ_{(q)} \\
				0 & \ZZ_{(q)}
		\end{array}\right)$, 
    and put $U_R = R/L$. 
    Then $U$ is uniserial 
    whose the endomorphism ring is not local (see. \cite{FS}). 
    According to \cite[Proposition 12.10]{AF}, 
    $U^2$ does not satisfy the FIEP. 
    We show $U^2$ is lifting. 
    For any nonzero homomorphism $f : U \to U/X$ 
    where $X$ is a submodule of $U$, 
    we can take 
    \[
    	f(\overline{\left(\begin{array}{@{\,}cc@{\,}}
				1 & 0 \\
				0 & 0
			\end{array}\right)})
        = \overline{\left(\begin{array}{@{\,}cc@{\,}}
				x & 0 \\
				0 & 0
			\end{array}\right)} + X
        \quad (x \in \ZZ_{(p)})
    \]
    If $x \in \ZZ_{(q)}$, 
    we can define a homomorphism $h : U \to U$ with 
    $h(\overline{\left(\begin{array}{@{\,}cc@{\,}}
				1 & 0 \\
				0 & 0
			\end{array}\right)})
        = \overline{\left(\begin{array}{@{\,}cc@{\,}}
				x & 0 \\
				0 & 0
			\end{array}\right)}$, 
    and $h$ satisfies $\pi h = f$, 
    where $\pi$ is the natural epimorphism from $U$ to $U/X$. 
    Otherwise we can express $x = p^m \frac{1}{q^n} \frac{t}{s}$, 
    where $m \in \NN \cup \{0\}$, $n \in \NN$ and
    $s, t \in \ZZ \setminus (p\ZZ \cup q\ZZ$). 
    Put $N = \overline{\left(\begin{array}{@{\,}cc@{\,}}
				p^m & 0 \\
				0 & 0
			\end{array}\right)}R$. 
    We can define an epimorphism 
    $h : N =  \to U$
    with $h(\overline{\left(\begin{array}{@{\,}cc@{\,}}
				p^m & 0 \\
				0 & 0
			\end{array}\right)})
        = \overline{\left(\begin{array}{@{\,}cc@{\,}}
				q^n \frac{s}{t} & 0 \\
				0 & 0
			\end{array}\right)}$, 
    and $h$ satisfies $fh = \pi |_N$, 
    where $\pi$ is the natural epimorphism from $U$ to $U/X$. 
    Therefore $U^2$ is lifting by Theorem \ref{lifting}. 
\end{exam}

\section*{Acknowledgements}
The author would like to thank Professors Yosuke Kuratomi and
Isao Kikumasa for valuable comments.
%%The authors are also grateful to the referee for the suggestions which improved the presentations of the paper.
%This work was supported by JSPS KAKENHI Grant Number 15K04821.
%\appendix
%
%
%%%%%%%%%%%%%%%%%%%%%%%%%%%%%%%%%%%%%%%%%%%%%%%%%%%%%%%%%%%%%%%%%%%%%%
%

\end{document}